\documentclass[letterpaper, 10 pt, conference]{ieeeconf}  

\IEEEoverridecommandlockouts                              
\newtheorem{theorem}{Theorem}
\newtheorem{corollary}{Corollary}
\newtheorem{lemma}{Lemma}
\newtheorem{remark}{Remark}
\newtheorem{proposition}{Proposition}

\usepackage{amsmath}
\usepackage{amsfonts}
\usepackage{graphicx}
\usepackage{amssymb}
\usepackage{subfigure}

\usepackage{wrapfig}
\usepackage{tikz}
\usetikzlibrary{arrows, automata, calc,decorations.markings}
\newcommand{\Cross}{$\mathbin{\tikz [x=1.4ex,y=1.4ex,line width=.2ex, red] \draw (0,0) -- (1,1) (0,1) -- (1,0);}$}

\title{\LARGE \bf
Target formation on the circle by monotone system design
}

\author{Cyrus Mostajeran, Jin Gyu Lee, Graham Van Goffrier and Rodolphe Sepulchre
\thanks{C. Mostajeran, J. G. Lee, and R. Sepulchre are with the Engineering Department of the University of Cambridge
        {\tt\small (csm54@cam.ac.uk)}}%
\thanks{G. Van Goffrier is with the Department of Physics and Astronomy, University College London
       }%
}

\begin{document}

\maketitle
\thispagestyle{empty}
\pagestyle{empty}

\begin{abstract}

Positivity and Perron-Frobenius theory provide an elegant framework for the convergence analysis of linear consensus algorithms. Here we consider a generalization of these ideas to the analysis of nonlinear consensus algorithms on the circle and establish tools for the design of consensus protocols that monotonically converge to target formations on the circle.
\end{abstract}

\section{INTRODUCTION}

Consensus problems and collective dynamics have been the subject of significant interest in the control community in recent decades with applications to cooperative and distributed control. Seminal works include \cite{Tsitsiklis1984,Jadbabaie2003,Olfati-Saber2004,Moreau2005}. See \cite{Dorfler2014} for a survey and  \cite{Moreau2005} for some examples of applications. More recently, there has been growing interest in the study of consensus algorithms defined on nonlinear spaces such as Lie groups \cite{Scardovi2007,Sarlette2010,Mostajeran2018}, the $n$-sphere \cite{Markdahl2018}, Grassmannians \cite{Sarlette2009}, and Stiefel manifolds \cite{Markdahl2020}. Consensus problems on nonlinear spaces
give rise to behaviours and global convergence issues that are not observed in linear models \cite{Sepulchre2011}. They are 
relevant for a number of engineering applications including the design of spatial coordinated motions \cite{Sepulchre2005}. Geometric consensus algorithms can be formulated intrinsically on a Riemannian manifold or extrinsically when the manifold is embedded in a Euclidean space . Most intrinsic consensus algorithms are based on the concepts of Riemannian distances, gradients, geodesics, and means. A fundamental challenge presented by consensus in nonlinear spaces is due to non-uniqueness of geodesics and topological properties of the underlying space, which result in problems that are fundamentally more complex and interesting than Euclidean analogues. 
Here, we will use an approach based on positivity theory and monotonicity to study consensus on the circle. In particular, we seek an answer to the following question: Can monotone system design be used to construct consensus algorithms that converge to a given target formation and collective motion on the circle?

Positivity theory plays an important role in the theory of dynamical systems with numerous applications to control engineering including to stabilization \cite{Farina2000,Leenheer2001}, observer design \cite{Hardin2007}, and distributed control \cite{Rantzer2011}, as well as the modelling of biological systems \cite{Angeli2004}. Linear positive systems are systems that leave a cone invariant. According to Perron-Frobenius theory, a linear system that is strictly positive, in the sense that it maps the boundary of a pointed convex solid cone into its interior, has a dominant one-dimensional eigenspace within the cone, which asymptotically attracts all trajectories inside the cone. A natural generalization of positivity to nonlinear systems is provided by the notion of \emph{differential positivity}, which is the property of systems whose linearizations along trajectories are positive with respect to a cone field \cite{Forni2015} and is closely related to monotonicity.
Strictly differentially positive systems infinitesimally contract a cone field along trajectories, constraining the asymptotic behavior to be one-dimensional under suitable technical conditions \cite{Forni2015,Mostajeran2018,Mostajeran2018monotonicity}. See \cite{Forni2019} for the closely related notion of $p$-dominance and its applications to differential dissipativity theory.

In this paper, we address the problem of designing consensus algorithms for a network of agents on the circle that converge to prescribed target formations. The solution is based on the design of a monotone system that guarantees convergence to a limit cycle corresponding to phase-locking behavior. In Section \ref{positivity and consensus}, we review the relevant technical background on linear positivity, consensus, and differential positivity. In Section \ref{consensus on circle}, we consider the application of differential positivity to consensus on the circle and describe conditions that guarantee convergence to phase-locked formations. In Section \ref{shaping}, we provide a solution to our main problem of designing systems that converge to prescribed target formations. We conclude with simulations and a brief discussion.

\section{Positivity, monotonicity, and consensus} \label{positivity and consensus}

\subsection{Linear positivity and consensus}

A linear system $\dot{x}=Ax$ on a vector space $\mathcal{V}$ is said to be positive with respect to a pointed convex solid cone $\mathcal{K}\subseteq \mathcal{V}$ if 
$e^{At}\mathcal{K}\subseteq\mathcal{K}$,
for all $t>0$, where $e^{At}\mathcal{K}:=\{e^{At}x:x\in\mathcal{K}\}$.
Continuous-time linear consensus algorithms take the form
$
\dot{x}=A(t)x,
$
where $A = (a_{ij})$ is a matrix whose rows sum to zero and whose off-diagonal elements are non-negative:
$A(t)\boldsymbol{1}=0$, and  $a_{ij}\geq 0$ for $i\neq j$.
Such continuous-time linear protocols arise from dynamics of the form
$\dot{x}_{i}=\sum_{j:(i,j)\in \mathcal{E}}a_{ij}(x_j-x_i)$
generated by $N$ agents exchanging information via a communication graph $\mathcal{G}$ with vertices and edges $(\mathcal{N}, \mathcal{E})$,
and are strictly positive with respect to the positive orthant $\mathcal{K}:=\mathbb{R}^N_+$ in $\mathbb{R}^N$ for a strongly connected graph. The projective distance to $\boldsymbol{1}$ given by the Hilbert metric of the positive orthant provides the Lyapunov function
\begin{equation*}
V(x)=\log{\max_i x_i \over \min_i x_i} = \max_i\log x_i - \min_i\log x_i,
\end{equation*}
which coincides with the well-known Tsitsiklis Lyapunov function in $\log$ coordinates. The Tsitsiklis Lyapunov function is non-increasing along solutions and the proof that it decreases strictly over a uniform horizon under appropriate assumptions can be established via elementary calculations \cite{Blondel2005}. The non-quadratic nature of the Tsitsiklis Lyapunov function is an essential feature of the convergence analysis of asymmetric and time-varying consensus algorithms. Indeed, \cite{Olshevsky2008} provides examples of matrices that satisfy the assumptions of a linear consensus algorithm but fail to admit a common time-invariant quadratic Lyapunov function.

\subsection{Differential positivity} 

Differential positivity can be defined on a manifold $\mathcal{M}$ equipped with a smooth cone field $\mathcal{K}(x)\subset T_x\mathcal{M}$. A continuous-time dynamical system $\Sigma$ is said to be differentially positive with respect to $\mathcal{K}$ if 
\begin{equation}
d\psi_t\vert_x\mathcal{K}(x)\subseteq\mathcal{K}(\psi_t(x)), \quad \forall x\in\mathcal{M}, \; \forall t\geq 0,
\end{equation}
where $\psi_t(x)$ is the flow at time $t$ from initial condition $x$ and $d\psi_t\vert_x$ denotes the differential of $\psi_t$ at $x$. The definition can be extended to strict differential positivity and uniformly strict differential positivity in obvious ways \cite{Mostajeran2018}.

Interestingly, differential positivity can be thought of as the local characterization of monotonicity \cite{Mostajeran2018,Mostajeran2018monotonicity}. Recall that a dynamical system $\Sigma$ on a vector space $\mathcal{V}$ endowed with a partial order $\preceq$ induced by some cone  $\mathcal{K}\subseteq\mathcal{V}$ is said to be monotone if for any $x_1, x_2\in\mathcal{V}$ the trajectories $\psi_{t}$ satisfy
$\psi_{t}(x_1)\preceq_{\mathcal{K}}\psi_{t}(x_2)$ whenever $x_1\preceq_{\mathcal{K}}x_2$,
for all $t > 0$. If $(x(\cdot),\delta x (\cdot))$ denotes a trajectory of the prolonged or variational system $\delta \Sigma$, then $\Sigma$ is monotone if and only if it is differentially positive.
In other words, the system is monotone if and only if for all $t > 0$, $\delta x(0)\in\mathcal{K} \Rightarrow \delta x(t)\in\mathcal{K}$. The infinitesimal characterization suggests a natural generalization to Lie groups, requiring differential positivity with respect to an invariant cone field \cite{Mostajeran2018,Mostajeran2018monotonicity}. 

In \cite{Forni2015}, the authors provide a generalization of Perron-Frobenius theory to nonlinear systems within a differential framework, whereby the Perron-Frobenius eigenvector of linear positivity theory is replaced by a Perron-Frobenius vector field $w(x)$ whose integral curves shape the attractors of the system. The main result on closed differentially positive systems is that the asymptotic behavior is either captured by a Perron-Frobenius curve $\gamma$ such that 
$\gamma'(s)=w(\gamma(s))$
at every point on $\gamma$ or is the union of the limit points of a trajectory that is nowhere aligned with the Perron-Frobenius vector field. For the purposes of this paper, the characterization provided by the following theorem will suffice \cite{Forni2015a}.

 \begin{theorem} \label{thm compact}
 Let $\Sigma$ be a uniformly strictly differentially positive system with respect to a cone field $\mathcal{K}(x)$ in a compact and forward invariant region $\mathcal{C}\subseteq\mathcal{M}$. If there exists a complete
 vector field $w$ satisfying $w(x)\in\operatorname{int}\mathcal{K}(x)\setminus\{0\}$ such that
$\limsup_{t\rightarrow\infty}|d\psi_t\vert_xw(x)|_{\psi_t(x)}<\infty$,
 and for all $x\in\mathcal{C}$ and $t\geq 0$:
 \begin{equation}
 w(\psi_t(x))={d\psi_t\vert_x w(x)\over|d\psi_t\vert_x w(x)|_{\psi_t(x)}},
 \end{equation}
 then there exists an integral curve of $w(x)$ whose image is an attractor for all the trajectories of $\Sigma$ from $\mathcal{C}$.
 \end{theorem}

\section{Consensus on the circle} \label{consensus on circle}

Consider a network of $N$ agents evolving on the circle $\mathbb{S}^1$ according to
\begin{equation} \label{circle consensus}
\dot{\theta}_i=\omega_i +\sum_{j:(i,j)\in\mathcal{E}}f_{ij}(\theta_j-\theta_i),
\end{equation}
where $\theta_i\in\mathbb{S}^1$ represents the phase of agent $i$, $\omega_i\in\mathbb{R}$ are prescribed `intrinsic' frequencies, and $f_{ij}$ denotes an odd coupling function on the domain $(-\pi,\pi)$ extended to $\mathbb{R}$ in such a way so as to make it $2\pi$-periodic. Note that $f_{ij}$ and $f_{ji}$ need not be the same function. Let $\theta=(\theta_1,\ldots,\theta_N)$ denote an element of the $N$-torus $\mathbb{T}^N$ and consider the $N$-tuple of vector fields $
\left(\frac{\partial}{\partial\theta^1},\ldots,\frac{\partial}{\partial\theta^N}\right)
$, which defines a basis of left-invariant vector fields on $\mathbb{T}^N$.
Assuming that the coupling functions $f_{ij}$ are differentiable and strictly monotonically increasing on $(-\pi,\pi)$, then it can be shown that the linearization $\dot{\delta\theta} = A(\theta)\delta\theta$ of the system given by (\ref{circle consensus}) is uniformly strictly differentially positive on the set $\mathbb{T}^N_{\pi}=\{\theta\in\mathbb{T}^N:|\theta_j-\theta_i|\neq \pi,\;(i,j)\in\mathcal{E}\}$ with respect to the cone field
\begin{equation*}
\mathcal{K}_{\mathbb{T}^N}(\theta):=\bigg\{\delta \theta\in T_{\theta}\mathbb{T}^N: \delta \theta^i\geq 0, \, \delta\theta=\sum_i\delta\theta^i\frac{\partial}{\partial\theta^i}\bigg\},
\end{equation*}
 for any strongly connected communication graph. Furthermore, the Perron-Frobenius vector field of the system on $\mathbb{T}^N_{\pi}$ is the left-invariant vector field $\boldsymbol{1}(\theta)=(1,\ldots,1)\in T_{\theta}\mathbb{T}^{N}$, where the vector representation is given with respect to the basis defined by $\left(\frac{\partial}{\partial\theta^1},\ldots,\frac{\partial}{\partial\theta^N}\right)$. Moreover, if we denote the flow of (\ref{circle consensus}) by $\psi_t$, then 
 the condition $A(\theta)\boldsymbol{1}(\theta)=0$ implies that $d\psi_t\vert_{\theta}\boldsymbol{1}_{\theta}=\boldsymbol{1}_{\psi_t(\theta)}$, which ensures that $\limsup_{t\rightarrow\infty} |d\psi_t\vert_{\theta} \boldsymbol{1}(\theta)|_{\psi_t(\theta)}<\infty$ for any flow confined to $\mathbb{T}^N_{\pi}$. If we add the requirement that the coupling functions $f_{ij}$ and $f_{ji}$ be barrier functions on $(-\pi,\pi)$ so that $f_{ij}(\alpha)\rightarrow \infty$ and $f_{ji}(\alpha)\rightarrow \infty$ as $\alpha\rightarrow \pi$, then the flow $\psi_t$ will be forward invariant on $\mathbb{T}^N_{\pi}$, resulting in the following theorem \cite{Mostajeran2018}.

\begin{theorem} \label{consensus theorem}
Consider a network of agents on $\mathbb{S}^1$ communicating via a connected communication graph according to (\ref{circle consensus}). If the coupling functions $f_{ij}$ satisfy $f_{ij}(0)=0$, $f_{ij}(\alpha)\rightarrow \infty$ as $\alpha \rightarrow \pi$, and $f_{ij}'(\alpha)>0$ on $(-\pi,\pi)$, then every trajectory from $\mathbb{T}_{\pi}^N$ converges to an integral curve of the vector field $\boldsymbol{1}=\boldsymbol{1}(\theta)$.
\end{theorem}

\begin{remark}
Convergence to an integral curve of $\boldsymbol{1}$ on $\mathbb{T}^N$ corresponds to phase-locking behavior, whereby the collective motion asymptotically converges to movement in a fixed formation with frequency synchronization among the agents. Further details may be found in \cite{Mostajeran2018}.
\end{remark}

\begin{remark} \label{remark: repel}
The coupling functions in Theorem \ref{consensus theorem} correspond to attractive couplings with barriers at $\pi$ separation of connected agents. Since differential positivity only requires $f'_{ij}(\alpha)>0$, a similar result would hold for coupling functions that are odd,  $2\pi$-periodic and differentiable on $(0,2\pi)$, and satisfy $f_{ij}(\pi)=0$, $f'_{ij}(\alpha)>0$ for $\alpha \in (0,2\pi)$ and $f_{ij}(\alpha)\rightarrow -\infty$ as $\alpha \rightarrow 0^+$. In this model, all agents $\vartheta_k$ repel each other with strengths that monotonically decrease on $(0,\pi)$ and grow infinitely strong as the separation between any pair of connected agents approaches $0$. See Figure \ref{fig:coupling} for examples of attractive and repulsive coupling functions with the specified properties.
Indeed, one would retain differential positivity by mixing attractive and repulsive couplings. Forward invariance would be guaranteed provided that both directions associated with each connection are of the same type, i.e. attractive or repulsive, resulting in barriers at separations of $0$ or $\pi$ between connected agents.
\end{remark}

\begin{figure}
\centering
\includegraphics[width=0.85\linewidth]{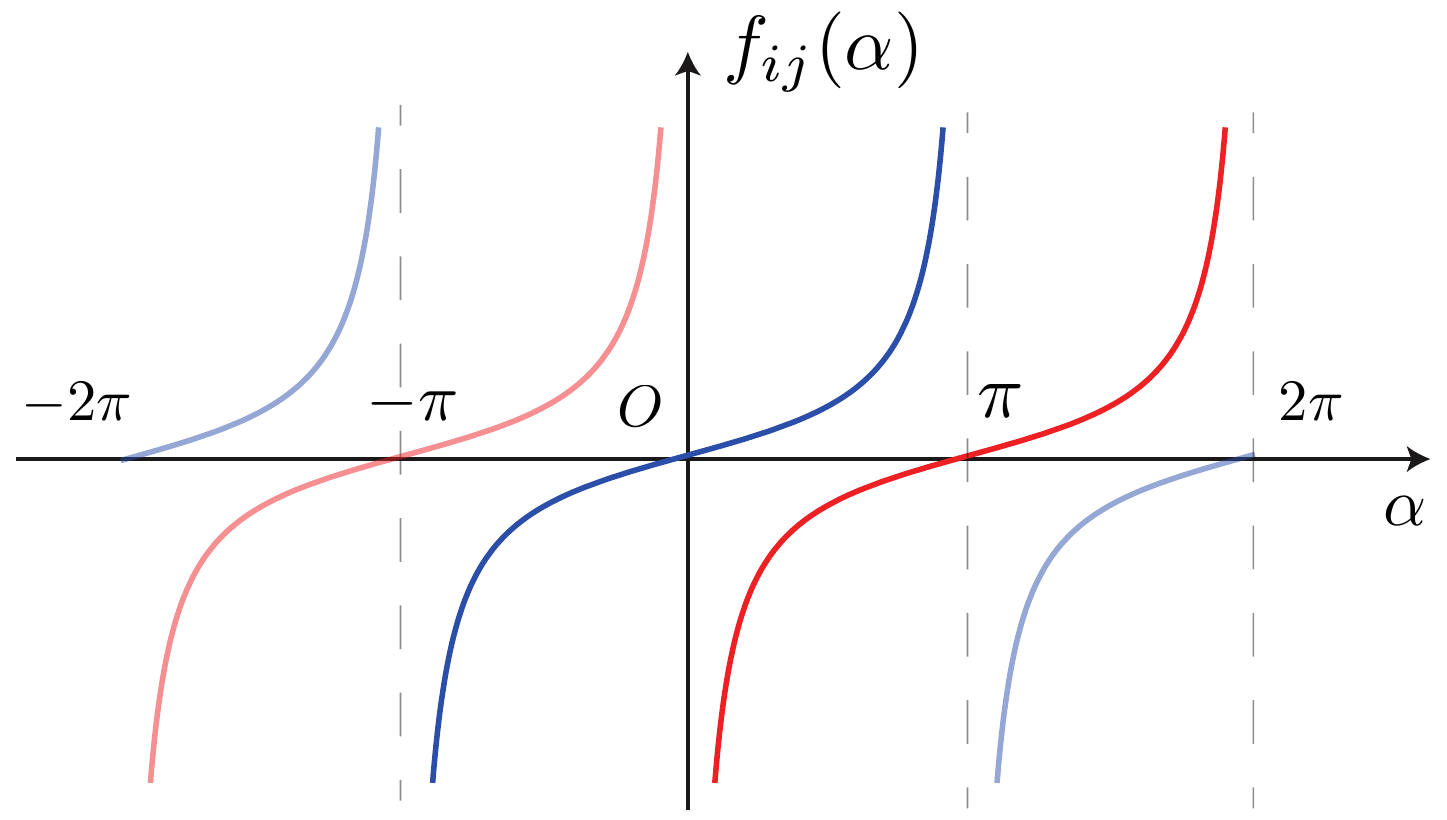}
  \caption{Attractive (blue) and repulsive (red) coupling functions $f_{ij}$ with the required monotonicity and barrier properties noted in Theorem \ref{consensus theorem} and Remark \ref{remark: repel}.
  }
  \label{fig:coupling}
\end{figure}

\subsection{Forward invariance and bidirectionality}

The limits $f_{ij}(\alpha)\rightarrow \infty$ as $\alpha \rightarrow \pi^-$ for attractive coupling functions and $f_{ij}(\alpha)\rightarrow -\infty$ as $\alpha \rightarrow 0^+$ for repulsive coupling functions are imposed to ensure forward invariance of differentially positive dynamics on the torus. We note however that this is only guaranteed if the connection between any pair of connected agents is bidirectional along the communication edge and of the same type. That is, we require for each $(i,j)\in\mathcal{E}$ that $\theta_i$ and $\theta_j$ interact via either attractive or repulsive couplings in either direction, even though this interaction need not be symmetric. If the communication graph is directed, we can construct examples where such coupling functions fail to ensure forward invariance in a region of differentially positive dynamics. This is best illustrated in the example of Figure \ref{fig:barrier}, where the agents are connected via a strongly connected directed graph and repel each other only in the direction of the arrows with a coupling function that grows infinitely strong at zero separation. If agents 3 and 4 move towards agent 1 with intrinsic frequencies of the opposite sign, they will cross the ``barrier" corresponding to agent 1. At the point of intersection, the infinite repulsions on agent 1 from agents 3 and 4 cancel, allowing the crossing to occur. Such pathological cases are avoided when the communication graph is bidirectional and of the same type along a given connection.

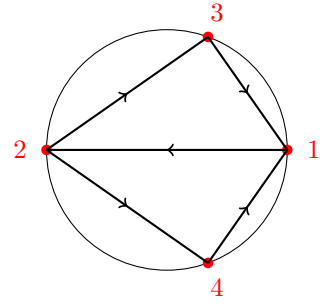
\begin{figure}
  \begin{center}
      \vspace{0cm}
    \begin{tikzpicture}[thick,decoration={
    markings,
    mark=at position 0.5 with {\arrow{>}}},
    ] 
        \coordinate (O) at (2,2);
        \def\radius{1.6cm}

        \draw[thin] (O) circle[radius=\radius];
        
        \path (O) ++(0:\radius) coordinate (a);
        \path (O) ++(180:\radius) coordinate (b);
        \path (O) ++(70:\radius) coordinate (c1);
        \path (O) ++(290:\radius) coordinate (c2);
        
        \fill[red] (a) circle[radius=2pt] ++(0:1em) node {$1$};
        \fill[red] (b) circle[radius=2pt] ++(180:1em) node {$2$};
        \fill[red] (c1) circle[radius=2pt] ++(70:1em) node {$3$};
        \fill[red] (c2) circle[radius=2pt] ++(290:1em) node {$4$};
        \draw[postaction={decorate}] (a) -- (b);
        \draw[postaction={decorate}] (b) -- (c1);
        \draw[postaction={decorate}] (c1) -- (a);
        \draw[postaction={decorate}] (b) -- (c2);
        \draw[postaction={decorate}] (c2) -- (a);

    \end{tikzpicture}
    \end{center}
    \vspace{-6pt}
    \caption{Directed connectivity graph for barrier-crossing example in which agents 3 and 4 cross agent 1 and consequently interchange positions.}
    \label{fig:barrier}
    \vspace{0cm}
\end{figure}

\subsection{Counting disconnected regions}


Repulsive and attractive coupling functions that ensure forward invariance of the consensus dynamics in $\mathbb{T}^N_0$ and $\mathbb{T}^N_{\pi}$, respectively, generally split the torus into a finite number of disconnected components determined by the communication graph. Given almost any initial configuration on the torus, the trajectory is attracted to a limit cycle
that is unique to the particular disconnected component corresponding to the initialization. Thus, even though the conditions of Theorem \ref{consensus theorem} guarantee almost global convergence to a limit cycle, the particular limit cycle may not be unique. The number of such behaviors is given by the number of disconnected components defined by the barrier functions and the communication graph topology. In this section, we address the problem of counting the number of such components.

Let $\mathcal{G}$ be an undirected communication graph with vertices and edges $(\mathcal{N}, \mathcal{E})$, where each $i \in \mathcal{N}$ represents an agent and an edge $(i,j) \in \mathcal{E}$ denotes that agents $i$ and $j$ are communicating; $|\mathcal{N}| = N$ and $|\mathcal{E}| = M$. Each agent $i$ has dynamics $\theta_i \in \mathbb{S}^1$, so we can define the configuration space of $\mathcal{G}$ to be the torus $\mathbb{T}^N$.

\subsubsection{Repulsive Case}

In the case of repulsive communication, coupling functions are required to grow infinitely strong at $0$ separation, so that for any edge $(i,j) \in \mathcal{E}$, $|\theta_i - \theta_j| \neq 0$. This restricts $\mathbb{T}^N$ to $\mathbb{T}^N_0 \equiv \{\theta \in \mathbb{T}^N : |\theta_i - \theta_j| \neq 0, \forall (i,j) \in \mathcal{E} \}$. Here we investigate the number of disconnected regions $R_0(\mathcal{G})$ in $\mathbb{T}^N_0$ for general graphs $\mathcal{G}$.


\begin{proposition}
For tree graphs $T$ and circular graphs $C_N$, $R_0(T) = 1$ and $R_0(C_N) = N-1$.
\end{proposition}

\begin{proof}
Consider first a single edge $(i,j)$ of $\mathcal{G}$, which places constraint ${|\theta_i - \theta_j| \neq 0}$. First we treat each coordinate as lying in $I^1 \equiv [0,2\pi]$ rather than in $\mathbb{S}^1$. The constraint produces two disconnected regions $\{+,-\} \subset I^2$, as shown in Figure \ref{fig:singlecons}, which merge to a single region after compactification to $\mathbb{T}^2$.


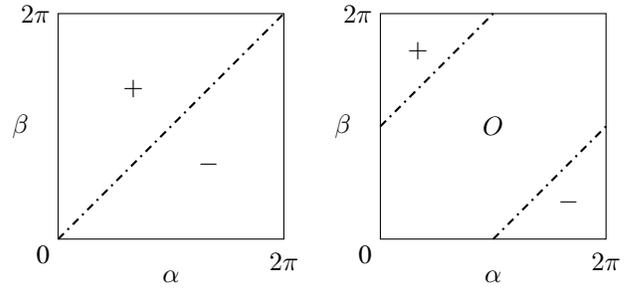
\begin{figure}
  \begin{center}
      \vspace{0cm}
    \subfigure{
        \begin{tikzpicture}
        \draw[postaction={decorate},decoration={
            markings,
            }
          ]
            (0,0) -- +(3,0) -- +(3,3) -- +(0,3) -- cycle;
        \draw [thick,dash dot] (0,0) -- (3,3);
        \node[] at (-0.2,-0.2) {0};
        \node[] at (-0.3,3) {$2\pi$};
        \node[] at (3,-0.3) {$2\pi$};
        
        \node[] at (1,2) {$+$};
        \node[] at (2,1) {$-$};
        
        \node[] at (-0.5,1.5) {$\bf{\beta}$};
        \node[] at (1.5,-0.5) {$\bf{\alpha}$};
        \end{tikzpicture}}
    \subfigure{
        \begin{tikzpicture}
        \draw[postaction={decorate},decoration={
            markings,
            }
          ]
            (0,0) -- +(3,0) -- +(3,3) -- +(0,3) -- cycle;
        \draw [thick,dash dot] (0,1.5) -- (1.5,3);
        \draw [thick,dash dot] (1.5,0) -- (3,1.5);
        \node[] at (-0.2,-0.2) {0};
        \node[] at (-0.3,3) {$2\pi$};
        \node[] at (3,-0.3) {$2\pi$};
        
        \node[] at (0.5,2.5) {$+$};
        \node[] at (1.5,1.5) {$O$};
        \node[] at (2.5,0.5) {$-$};
        
        \node[] at (-0.5,1.5) {$\bf{\beta}$};
        \node[] at (1.5,-0.5) {$\bf{\alpha}$};
        \end{tikzpicture}}
    \end{center}
    \caption{Single edge constraints between two agents, for repulsive (left) and attractive coupling (right). $M$ such constraints together have the effect of carving up $\mathbb{T}^N_0$ and $\mathbb{T}^N_\pi$, respectively, into disconnected regions.}
    \label{fig:singlecons}
    \vspace{0cm}
\end{figure}

The full connectivity graph puts $M$ such constraints on $\mathbb{T}^N$, or $I^N$ prior to compactification. Any point in $I^N$ then has a well-defined label $\{+,-\}^M$, enumerated by some ordering of the graph edges. Let us assume for the moment that all such labellings are permitted, and correspond to physically-realizable configurations of agents. $I^N$ then contains $2^M$ unique disconnected regions. However, for compactification to proceed, each of the $N$ pairs of $(N-1)$-faces corresponding to $\theta_i = 0$ and $\theta_i = 2\pi$ must be identified, which may cause previously disconnected regions to become connected. In our graph labelling, this identification corresponds to reversing the labelling of all edges adjoined to $\theta_i$, as the barriers for each of these edges must all be crossed to travel from $0$ to $2\pi$ or vice-versa, while all other edge barriers run perpendicular to the $(N-1)$-face in question. 

The set of $(N-1)$-face identification operators may be treated as members of an $M$-dimensional vector space over the field with two elements, $\mathbb{F}_2^M$, which span some subspace $H \subseteq \mathbb{F}_2^M$. If all identification operators could be freely applied to all graph labellings, $H$ would contain all edge labellings which are path-connected to the zero-label $O^M$, and $D = \mathbb{F}_2^M / H$ would be isomorphic to the set of equivalence classes of disconnected labellings. From this and the linear dependence identity $\sum_i\hat{A}_i = 0^M$, the simple relation $R_0(\mathcal{G}) = |D| = 2^{M-N+1}$ would hold. Unfortunately, each graph labelling permits only a restricted subset of identification operators on its vertices, corresponding to those agents in the configuration which can be translated across $\theta=0$ without colliding with another edge-adjacent agent (note that two agents which are not in communication may collide without consequence).

To address this, we must first address our unfounded assumption from earlier in this argument, that all connectivity graph labellings $\{+,-\}^M$ were permitted. This was excessively bold, as many labellings describe arrangements of agents which are not self-consistently ordered. This is most easily seen by assigning an orientation to our graph in place of a two-labelling, where the directed arrow of an edge points from $\alpha \rightarrow \beta$ if $\theta_\alpha < \theta_\beta$. The simplest such example is shown in Figure \ref{fig:illegalrep}, where angle $\theta_\gamma$ cannot be placed on the circle such that it satisfies both $\theta_\gamma > \theta_\beta$ and $\theta_\gamma < \theta_\alpha$. Of course, vertex $\gamma$ is not responsible for the problem -- reversing the direction of any edge would make this graph permissible.

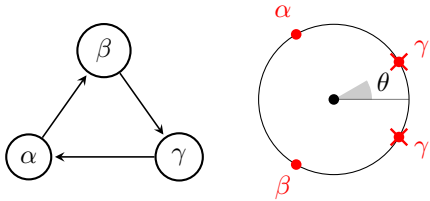
\begin{figure}
  \begin{center}
      \vspace{0cm}
    \subfigure{
    \raisebox{1em}{
    \begin{tikzpicture}[
            > = stealth, 
            shorten > = 1pt, 
            auto,
            node distance = 2cm, 
            semithick 
        ]

        \tikzstyle{every state}=[
            draw = black,
            thick,
            fill = white,
            minimum size = 2mm
        ]

        \node[state] (a) at (0,0) {$\alpha$};
        \node[state] (b) at (1,1.414) {$\beta$};
        \node[state] (c) at (2,0) {$\gamma$};

        \path[->] (a) edge node {} (b);
        \path[->] (b) edge node {} (c);
        \path[->] (c) edge node {} (a);
    \end{tikzpicture}}}
    \hspace{1em}
    \subfigure{
    \begin{tikzpicture}
        \coordinate (O) at (2,2);
        \def\radius{1cm}
        
        \draw (O) circle[radius=\radius];
        \fill (O) circle[radius=2pt] node[below left]{};
        
        \path (O) ++(120:\radius) coordinate (a);
        \path (O) ++(240:\radius) coordinate (b);
        \path (O) ++(330:\radius) coordinate (c1);
        \path (O) ++(30:\radius) coordinate (c2);
        
        \fill[red] (a) circle[radius=2pt] ++(120:1em) node {$\alpha$};
        \fill[red] (b) circle[radius=2pt] ++(240:1em) node {$\beta$};
        \fill[red] (c1) circle[radius=2pt] ++(330:1em) node {$\gamma$};
        \fill[red] (c2) circle[radius=2pt] ++(30:1em) node {$\gamma$};
        \node at (c1)  {\Cross};
        \node at (c2)  {\Cross};
        
        \path (O) ++(0:\radius) coordinate (x);
        \path (O) ++(30:\radius) coordinate (y);
        \begin{scope}
        \path[clip] (O) -- (x) -- (y);
        \draw[thin] (O) -- (x);
        \fill[black, opacity=0.2, draw=black] (O) circle (5mm);
        \node at ($(O)+(20:7mm)$) {$\theta$};
        \end{scope}
    \end{tikzpicture}}
    \vspace{-1em}
    \end{center}
    \caption{A simple illegal graph labelling in the repulsive coupling case: agent $\gamma$ cannot simultaneously satisfy $\theta_\gamma > \theta_\beta$ and $\theta_\gamma < \theta_\alpha$ if $\theta_\beta > \theta_\alpha$.}
    \label{fig:illegalrep}
    \vspace{0cm}
\end{figure}

In order to avoid such clashes in ordering, the orientation of the graph is required to be acyclic. The enumeration of acyclic orientations is solved for all graphs, and is given by $N_A(\mathcal{G}) = |\chi_{\mathcal{G}}(-1)| = 2^M - N_C(\mathcal{G})$, where $\chi_{\mathcal{G}}$ is the chromatic polynomial of $\mathcal{G}$ \cite{Stanley2006}. While $\chi_G$ is known and takes simple forms for many graphs, it must be computed algorithmically for unknown graphs, and this problem is \#P-hard \cite{Linial1986}.

We may now make use of orientations to characterize the set of identification operators permissible for each graph labelling. Any vertex $\alpha$ which is unobstructed from crossing $\theta=0$ must have either $\theta_\alpha < \theta_{\beta_i}$ or $\theta_\alpha > \theta_{\beta_i}$ for all edge-connected vertices $\beta_i$. In an acyclic orientation, such vertices are known as sources and sinks; for every acyclic orientation of any graph, $N_{sources} \geq 1$ and $N_{sinks} \geq 1$, and orientations with one of each are known as bipolar orientations. Therefore, only the set of identification operators at the source and sink vertices of a graph orientation may be applied to that labelling. Equivalence classes of labellings may be constructed by applying the permitted operators to all acyclic orientations of $\mathcal{G}$; the quantity of equivalence classes is then $R_0(\mathcal{G})$.

For tree graphs $T$, the situation is fortunately simple: all orientations must be acyclic, so $N_C(T) = 0$. Furthermore, all orientations lie in the same equivalence class. The outermost vertices of the tree must always be source or sinks, and so their identification operator may be freely applied to flip the direction of the outermost edges. This, in turn, allows all second-to-outermost vertices to be transformed into sources or sinks, and by iteration, the entire tree may be freely reoriented. Thus $R_0(T) = 1$.

\begin{lemma}
    For circle graphs $C_N$, all orientations with $P$ clockwise-directed edges and $N-P$ counterclockwise-directed edges form a unique, closed equivalence class under source-sink identification operators.
\end{lemma}
\begin{proof}
    A vertex lies between a clockwise- and a counterclockwise-directed edge if and only if it is a source or a sink; without loss of generality, consider a source. Acting on this vertex with its identification operator can be thought of as shifting the clockwise-directed edge in the counterclockwise direction along $C_N$, and vice versa. If a neighboring vertex was previously a sink, it will become neutral; otherwise, it will become a source. Therefore, any orientation in the class can be transformed into any other desired orientation by moving all clockwise-directed edges into the desired position, one step at a time. Furthermore, by the binomial theorem, there are ${N \choose P}$ labellings in this equivalence class.
\end{proof}

Enumerating by the number of clockwise edges, and rejecting the $N_C(C_N) = 2$ cyclic orientations, we have $R_0(C_N) = N-1$. The partitioning of all graph labellings is confirmed by noting that $\sum_{P=0}^N {N \choose P} = 2^N$. 

\end{proof}

\begin{proposition}
    For complete graphs $K_N$, $R_0(K_N) = (N-1)!$.
\end{proposition}
\begin{proof}
    As all vertices are adjacent in a complete graph, no agent may collide with any other agent. Therefore $R_0(K_N)$ reduces to the number of distinct circular permutations of $N$ items, which is $(N-1)!$. We may also interpret this case through the above source-sink formalism. Acyclic orientations of complete graphs are always bipolar \cite{deFraysseix1995}, and the source of the orientation corresponds to the agent with smallest $\theta$. Acting on this source vertex with its identification operator causes it to become a sink, and the next smallest-$\theta$ vertex to become the new source; iterating, every equivalence class under the identification operators must contain $N$ labellings. Furthermore, $N_A(K_N) = \chi_{K_N}(-1) = N!$, so $R_0(K_N) = N! / N = (N-1)!$.
\end{proof}

\subsubsection{Attractive Case}

In the case of attractive communication, coupling functions instead must satisfy $|\theta_i - \theta_j| \neq \pi$, restricting $\mathbb{T}^N$ to $\mathbb{T}^N_\pi \equiv \{\theta \in \mathbb{T}^N : |\theta_i - \theta_j| \neq 0, \forall (i,j) \in \mathcal{E} \}$. In this case, a single constraint produces three disconnected regions $\{+,O,-\} \subset I^2$, as shown in Figure \ref{fig:singlecons}. We do not attempt here a complete discussion of the attractive case, but point out some qualitative differences in treatment from the repulsive case. 

A formalism of edge-labellings and vertex compactification operators, as was developed in the repulsive case, may once again be applied. It is suspected but unconfirmed that this labelling may be reduced to an oriented two-coloring of the undirected graph, as $+$ and $-$ labellings may not appear edges adjacent to the same vertex. Even within this scheme, it appears that the $\pm$-colored subgraph must be connected and bipartite in order to correspond to a physical formation, leading to a complex counting problem for general graphs, even before compactification is considered.

\section{Shaping consensus to a target formation} \label{shaping}




Now, we consider the problem of shaping consensus; given a connected graph $\mathcal{G} = (\mathcal{N}, \mathcal{E})$, a set of intrinsic frequencies $\{\omega_i\}_\mathcal{N}$, and a formation $\{\Delta_{ij}\}_\mathcal{E}$, i.e., a phase difference given for each edge, find a common frequency $\bar{\omega}$, a set of attractive undirected edges $\mathcal{E}^a$, a set of repulsive undirected edges $\mathcal{E}^r$, a set of attractive coupling functions $\{f_{ij}(\cdot)\}_{\mathcal{E}^a}$, and a set of repulsive coupling functions $\{g_{ij}(\cdot)\}_{\mathcal{E}^r}$ such that the given formation represents a limit cycle governed by the equation
\begin{align}\label{eq:lim_cyc}
\bar{\omega} = \omega_i + \sum_{j \in \mathcal{N}_i^a} f_{ij}(\Delta_{ij}) + \sum_{j \in \mathcal{N}_i^r} g_{ij}(\Delta_{ij}), \quad i \in \mathcal{N},
\end{align}
to which the system described by
\begin{align}
\begin{split}
\dot{\theta}_i(t) &=\omega_i +  \sum_{j \in \mathcal{N}_i^a} f_{ij}\left(\theta_j(t) - \theta_i(t)\right)\\
&\quad + \sum_{j \in \mathcal{N}_i^r} g_{ij}\left(\theta_j(t) - \theta_i(t)\right), \quad i \in \mathcal{N}
\end{split}
\end{align}
converges. Here, $\mathcal{N} := \{1, \dots, N\}$, $\mathcal{N}_i := \{j \in \mathcal{N} : (i, j) \in \mathcal{E}\}$, $\mathcal{N}_i^a$ and $\mathcal{N}_i^r$ denote the subsets of nodes connected to agent $i$ by attractive and repulsive couplings, respectively, and $f_*(\cdot) : \mathbb{R} \to \mathbb{R}$, $g_*(\cdot) : \mathbb{R} \to \mathbb{R}$ are $2\pi$-periodic functions that are twice differentiable on $(-\pi, \pi)$, $(0, 2\pi)$ respectively such that $f_*(0) = 0$, $g_*(\pi) = 0$, $f_*'(\alpha) > 0$ for all $\alpha \in (-\pi, \pi)$, and $g_*'(\alpha) > 0$ for all $\alpha \in (0, 2\pi)$.
In particular, 
\begin{align*}
\lim_{t \to (2n+1)\pi^+} f_*(t) &= -\infty \quad \text{and} \quad \lim_{t \to (2n+1)\pi^-} f_*(t) = \infty \\
\lim_{t \to 2n\pi^+} g_*(t) &= -\infty \quad \text{and} \,\,\,\,\quad\quad \lim_{t \to 2n\pi^-} g_*(t) = \infty
\end{align*}
for all $n \in\mathbb{N}$.

In this section, we first present a necessary and sufficient condition for this problem to be solvable.
Then, we illustrate how to choose a subgraph that minimizes the number of connections.
Finally, by restricting the coupling functions to take a prototypical shape, we derive a minimum energy solution.
An additional remark on using only the attractive coupling is also provided.

Now, first note that a necessary condition for this problem to be solvable is simply that there exists no edge $(i, j) \in \mathcal{E}$ such that $\Delta_{ij} = 0,\pi \mod 2\pi$.
This comes from our freedom to choose between attractive and repulsive coupling for each edge.
Our claim is that this is also sufficient.

\begin{theorem}\label{thm:NS}
Given a connected graph $\mathcal{G} = (\mathcal{N}, \mathcal{E})$, a set of intrinsic frequencies $\{\omega_i\}_\mathcal{N}$, and a formation $\{\Delta_{ij}\}_\mathcal{E}$, assume that $\Delta_{ij} \neq 0, \pi \mod 2\pi$ for all $(i, j) \in \mathcal{E}$.
Then, there exists a common frequency $\bar{\omega}$, a set of attractive undirected edges $\mathcal{E}^a$, a set of repulsive undirected edges $\mathcal{E}^r = \mathcal{E} \setminus \mathcal{E}^a$, a set of attractive coupling functions $\{f_{ij}(\cdot)\}_{\mathcal{E}^a}$, and a set of repulsive coupling functions $\{g_{ij}(\cdot)\}_{\mathcal{E}^r}$ such that~\eqref{eq:lim_cyc} is satisfied. 
\end{theorem}

\begin{proof}
Note first that if we can find a common frequency $\bar{\omega}$ and a set $\mathcal{E}^a$ such that for each $i \in \mathcal{N}$, there exists $j \in \mathcal{N}_i^a := \{j \in \mathcal{N}_i : (i, j) \in \mathcal{E}^a\}$ satisfying
$$\Delta_{ij} \in \begin{cases} (0, \pi) \,\,\,\,\,\mod 2\pi, &\mbox{ if } \bar{\omega} > \omega_i, \\ (-\pi, 0) \mod 2\pi, &\mbox{ if } \bar{\omega} < \omega_i,\end{cases}$$
or $j \in \mathcal{N}_i^r := \{j \in \mathcal{N}_i : (i, j) \in \mathcal{E}^r:=\mathcal{E}\setminus\mathcal{E}^a\}$ satisfying
$$\Delta_{ij} \in \begin{cases} (0, \pi) \,\,\,\,\mod 2\pi, &\mbox{ if } \bar{\omega} < \omega_i, \\ (\pi, 2\pi) \mod 2\pi, &\mbox{ if } \bar{\omega} > \omega_i,\end{cases}$$
then we can find a set of attractive coupling functions $\{f_{ij}(\cdot)\}_{\mathcal{E}^a}$ and a set of repulsive coupling functions $\{g_{ij}(\cdot)\}_{\mathcal{E}^r}$ such that~\eqref{eq:lim_cyc} is satisfied.
To see this note that if, for instance, such $j \in \mathcal{N}_i^a$ exists, then we can always find $f_{ik}(\cdot)$ for $k \in \mathcal{N}_i^a$, $k \neq j$ and $g_{ik}(\cdot)$ for $k \in \mathcal{N}_i^r$ such that $f_{ik}(\Delta_{ik})$ and $g_{ik}(\Delta_{ik})$ are sufficiently small so that
$$\left|\bar{\omega} - \omega_i\right| > \sum_{k \in \mathcal{N}_i^a, k \neq j} \left|f_{ik}(\Delta_{ik})\right| + \sum_{k \in \mathcal{N}_i^r}\left| g_{ik}(\Delta_{ik})\right|$$
and therefore, we can find $f_{ij}(\cdot)$ such that 
$$f_{ij}(\Delta_{ij}) = \bar{\omega} - \omega_i - \sum_{k \in \mathcal{N}_i^a, k\neq j} f_{ik}(\Delta_{ik}) - \sum_{k \in \mathcal{N}_i^r} g_{ik}(\Delta_{ik}).$$

Now, we can show that for any common frequency $\bar{\omega} \in (\min_i\omega_i, \max_i \omega_i)$ such that $\bar{\omega} \neq \omega_i$, $i \in \mathcal{N}$, there exists a set $\mathcal{E}^a$ that satisfies the above condition.
In particular, pick any such common frequency $\bar{\omega}$ and let $\mathcal{N}_\text{large} := \{i \in \mathcal{N}: \omega_i > \bar{\omega}\}$ and $\mathcal{N}_\text{small} := \{i \in \mathcal{N} : \omega_i < \bar{\omega}\}$.
Now, for each subgraph that corresponds to $\mathcal{N}_\text{large}$ and $\mathcal{N}_\text{small}$, there will be multiple connected components.
But, these should be connected to at least one of the connected components on the other side.
So, for all of these edges, say $(i, j) \in \mathcal{E}$, $i \in \mathcal{N}_\text{small}$, $j \in \mathcal{N}_\text{large}$, if $\Delta_{ij}\in (0, \pi) \mod 2\pi$, then let $(i, j) \in \mathcal{E}^a$.
Otherwise, if $\Delta_{ij} \in (\pi, 2\pi) \mod 2\pi$, then let $(i, j) \in \mathcal{E}^r$.
By this construction, those agents in $\mathcal{N}_\text{small}$, that have neighbors in $\mathcal{N}_\text{large}$, satisfy the above condition, and vice versa for those agents in $\mathcal{N}_\text{large}$.

For the remaining edges that are associated with each connected component, if this corresponds to $\mathcal{N}_\text{small}$, then starting from the agents that have a connection with the other side, define as above the attractiveness and the repulsiveness by considering the starting agents as $\mathcal{N}_\text{large}$ and their neighbors as $\mathcal{N}_\text{small}$, and repeat.
Analogously, for connected components corresponding to $\mathcal{N}_\text{large}$, start from the agents that have a connection with the other side and define as above the attractiveness and the repulsiveness by considering the starting agents as $\mathcal{N}_\text{small}$ and their neighbors as $\mathcal{N}_\text{large}$, and repeat.
In this way, all of the agents satisfy the above condition.
\end{proof}

Now, having shown that we have a mild necessary and sufficient condition, a direct conclusion can be made about the choice of a subgraph that minimizes the number of connections.

\begin{corollary}
Given a connected graph $\mathcal{G} = (\mathcal{N}, \mathcal{E})$, any undirected spanning tree of the original graph $\mathcal{G}$ has the minimum number of edges among all the subgraphs that still allow the problem to be solvable.
\end{corollary}

On the other hand, under this necessary and sufficient condition, given a common frequency $\bar{\omega}$, a set of attractive undirected edges $\mathcal{E}^a$, and a set of repulsive undirected edges $\mathcal{E}^r = \mathcal{E} \setminus\mathcal{E}^a$, if we restrict our coupling functions to be the scaled version of a single function, i.e., $f_{ij}(\cdot) = \alpha_{ij}f^*(\cdot)$ and $g_{ij}(\cdot) = \beta_{ij}g^*(\cdot)$, then the problem of finding an appropriate set of coupling functions becomes algebraic, as $f^*(\Delta_{ij})$ and $g^*(\Delta_{ij})$ are constants.
In particular, we only have to find real $\{\alpha_{ij} > 0\}_{\mathcal{N}_i^a}$, $\{\beta_{ij} > 0\}_{\mathcal{N}_i^r}$ such that
\begin{align}\label{eq:alg_lim_cyc}
\bar{\omega} - \omega_i = \sum_{j \in \mathcal{N}_i^a} \alpha_{ij}f_{ij}^* + \sum_{j \in \mathcal{N}_i^r} \beta_{ij} g_{ij}^*
\end{align}
for each $i \in \mathcal{N}$, where $f_{ij}^* := f^*(\Delta_{ij})$ and $g_{ij}^* := g^*(\Delta_{ij})$.

In this respect, a minimum energy solution, in the sense of minimizing $\sum_{i \in \mathcal{N}}\sum_{j \in \mathcal{N}_i^a}\alpha_{ij}^2 + \sum_{j \in \mathcal{N}_i^r} \beta_{ij}^2$, can be found as follows.

\begin{theorem}\label{thm:min_energ}
Under the necessary and sufficient condition, let us assume that a common frequency $\bar{\omega}$, a set of attractive undirected edges $\mathcal{E}^a$, and a set of repulsive undirected edges $\mathcal{E}^r = \mathcal{E}\setminus \mathcal{E}^a$ are given such that for each $i \in \mathcal{N}$ the index set $\mathcal{N}_i^+ := \mathcal{N}_i^{a,+} \cup \mathcal{N}_i^{r,+}$ is non-empty, where 
\begin{align*}
\mathcal{N}_i^{a,+} &:= \{j \in \mathcal{N}_i^a : \text{sgn}(\bar{\omega} - \omega_i) = \text{sgn}(f_{ij}^*) \},\\
\mathcal{N}_i^{a,-} &:= \{j \in \mathcal{N}_i^a : \text{sgn}(\bar{\omega} - \omega_i) = -\text{sgn}(f_{ij}^*) \},\\
\mathcal{N}_i^{r,+} &:= \{j \in \mathcal{N}_i^r : \text{sgn}(\bar{\omega} - \omega_i) = \text{sgn}(g_{ij}^*) \},\\
\mathcal{N}_i^{r,-} &:= \{j \in \mathcal{N}_i^r : \text{sgn}(\bar{\omega} - \omega_i) = -\text{sgn}(g_{ij}^*) \}.
\end{align*}
Then, we have a set $\{\alpha_{ij}>0\}_{\mathcal{N}_i^a}$ and $\{\beta_{ij} > 0\}_{\mathcal{N}_i^r}$ such that~\eqref{eq:alg_lim_cyc} is satisfied.
Among which, an almost minimum energy solution with arbitrary precision can be found as
\begin{align*}
\alpha_{ij} &= \begin{cases} \epsilon, &\mbox{ if } j \in \mathcal{N}_i^{a,-}, \\  \frac{\left|\bar{\omega} - \omega_i\right|\left|f_{ij}^*\right|}{\sum_{j \in \mathcal{N}_i^{a,+}} \left(f_{ij}^*\right)^2 + \sum_{j \in \mathcal{N}_i^{r,+}}\left(g_{ij}^*\right)^2}, &\mbox{ if } j\in \mathcal{N}_i^{a,+},\end{cases}\\
\beta_{ij} &= \begin{cases} \epsilon, &\mbox{ if } j \in \mathcal{N}_i^{r,-}, \\  \frac{\left|\bar{\omega} - \omega_i\right|\left|g_{ij}^*\right|}{\sum_{j \in \mathcal{N}_i^{a,+}} \left(f_{ij}^*\right)^2 + \sum_{j \in \mathcal{N}_i^{r,+}}\left(g_{ij}^*\right)^2}, &\mbox{ if } j\in \mathcal{N}_i^{r,+},\end{cases}
\end{align*}
with arbitrarily small $\epsilon > 0$. 
\end{theorem}

\begin{proof}
Note that the increase of $\alpha_{ij}$ and $\beta_{ij}$ for $j \in \mathcal{N}_i \setminus \mathcal{N}_i^+$ necessarily increases the energy, as it acts in opposition to the goal of satisfying~\eqref{eq:alg_lim_cyc}.
Now, we have the following simple inequality that proves our claim.
\begin{align*}
\left|\bar{\omega} - \omega_i\right|^2 &=\left( \sum_{j \in \mathcal{N}_i^a} \alpha_{ij}f_{ij}^* + \sum_{j \in \mathcal{N}_i^r} \beta_{ij} g_{ij}^*\right)^2 \\
&\le \left(\sum_{j \in \mathcal{N}_i^{a,+}} \alpha_{ij}f_{ij}^* + \sum_{j \in \mathcal{N}_i^{r,+}} \beta_{ij}g_{ij}^*\right)^2 \\
&\le  \left(\sum_{j \in \mathcal{N}_i^{a,+}} \alpha_{ij}^2 +\sum_{j \in \mathcal{N}_i^{r,+}}\beta_{ij}^2\right) \times \\
&\quad\quad\quad\quad\quad\quad \left(\sum_{j \in \mathcal{N}_i^{a,+}} \left(f_{ij}^*\right)^2 + \sum_{j \in \mathcal{N}_i^{r,+}} \left(g_{ij}^*\right)^2\right)
\end{align*}
where the first inequality follows for sufficiently small $\epsilon > 0$, and the second inequality is simply the consequence of the Cauchy-Schwarz inequality, where only the proposed solution satisfies the equality. 
\end{proof}

\subsection{Using only attractive coupling}

If we are restricted to use only attractive couplings, then some minor changes apply to our solution.
First of all, our necessary condition becomes stronger.
In particular, a necessary condition for this problem to be solvable is that 1) there exists no edge $(i, j) \in \mathcal{E}$ such that $\Delta_{ij} = \pi \mod 2\pi$ and that 2) there exists $\bar{\omega}$ such that for each $i \in \mathcal{N}$ the index set $\mathcal{N}_i^+$ is non-empty, where $\mathcal{N}_i^+$ is defined as 
$$\left\{ j \in \mathcal{N}_i : \Delta_{ij} \in \begin{cases} (0, \pi) \,\,\,\,\,\mod 2\pi, &\mbox{ if } \bar{\omega} > \omega_i, \\ (-\pi, 0) \mod 2\pi, &\mbox{ if } \bar{\omega} < \omega_i.\end{cases}\right\}.$$
This is also sufficient according to the proof of Theorem~\ref{thm:NS}.

On the other hand, the choice of a subgraph that minimizes connections becomes more complicated as follows.
\begin{theorem}
Given a connected graph $\mathcal{G} = (\mathcal{N}, \mathcal{E})$, a set of intrinsic frequencies $\{\omega_i\}_\mathcal{N}$, and a formation $\{\Delta_{ij}\}_\mathcal{E}$, assume that there exists a common frequency $\bar{\omega}$ such that for each $i \in \mathcal{N}$ the index set $\mathcal{N}_i^+$ is non-empty.
Then, the minimum number of edges of a subgraph $\mathcal{G}'$ that still allows the problem to be solvable is
$$N + C - 1 + \sum_{k=1}^C (n_k - 1)$$
where $C$ is the number of connected components of a graph $\bar{\mathcal{G}} = (\mathcal{N}, \bar{\mathcal{E}})$, $\bar{\mathcal{E}} := \{(i, j) \in \mathcal{E} : j \in \mathcal{N}_i^+ \text{ or } i \in \mathcal{N}_j^+\}$ and $n_k$, $k = 1, \dots, C$ is the number of independent strongly connected components (iSCCs) of a directed graph $\mathcal{G}^k = (\mathcal{N}^k, \mathcal{E}^k)$, where $\mathcal{N}^k$ corresponds to the indices of each connected component of $\bar{\mathcal{G}}$ and $\mathcal{E}^k := \{(i, j) \in \mathcal{E} : j \in \mathcal{N}_i^+, i \in \mathcal{N}^k\}$. 
\end{theorem}

\begin{proof}
First of all, note that to satisfy the necessary and sufficient condition for the subgraph $\mathcal{G}'$, for each $i \in \mathcal{N}$, we have to select at least one edge $(i, j)$ such that $j \in \mathcal{N}_i^+$.
Moreover, we have to select additional edges to maintain connectivity of the network $\mathcal{G}'$.

In this sense, it is necessary to have these connected components connected by an edge contained in $\mathcal{E} \setminus \bar{\mathcal{E}}$.
Therefore, we additionally need at least $C-1$ edges.
Now, for each connected component, we should find a subgraph with a minimal number of connections that satisfies the above conditions.
To achieve this, the subgraph for each connected component should at least have an edge for each of its elements, combined with edges connecting its $n_k$ iSCCs, for a total no less than $|\mathcal{N}^k| + n_k - 1$.
Therefore, we prove our claim, so long as we can construct a subgraph that has exactly the number of edges described here.

Since the original graph $\mathcal{G}$ is connected, we can always find such $C-1$ edges.
In particular, if we construct a new graph, where each connected component is considered as a single node, then it should be connected, and thus, we can find an undirected spanning tree.
We just have to pick any single edge that corresponds to an edge of this spanning tree.

Now, for each connected component, 
a network with a minimal number of connections can be found as follows.
Take any agent in the iSCC, and find a spanning tree that has its root as the agent, then simply connect one edge coming from any other agent to the root agent.
For the remaining agents, we construct a network step by step.
First, we connect all agents that have as neighbor a member of any iSCC with only one edge per agent. Then, we connect all the agents that have their neighbor as a member of this first layer by only one edge per agent and repeat.
If there is a remaining agent, then it is not connected to any of the agents thus far.
But, this means that when considering the undirected graph that is obtained by making the edges symmetric, it is no longer a single connected component, which is a contradiction.
\end{proof}

The minimum energy solution is identical to what was presented in Theorem~\ref{thm:min_energ}.

\subsection{Simulations}

Simulations of consensus dynamics on the circle were performed in C++, using the classic Runge-Kutta method for discrete time steps and prototype attractive coupling function $f_*(\theta) = \tan(\theta/2)$.
The performance of least communication solutions are demonstrated here for tree graphs containing seven agents, and with attractive couplings only. A ``balanced'' target formation is constructed with all neighboring $\Delta_{ij} = 2\pi/7$, and a ``clustered'' formation is defined with $\Delta_{12} = 0.1$, $\Delta_{23} = 1.3$, $\Delta_{34} = \Delta_{45} = 0.2$, $\Delta_{56} = 1.3$, $\Delta_{67} = 0.1$, both of which are represented in Figure \ref{fig:formations}. Pseudorandom initial positions were chosen, in accordance with the almost-everywhere convergence property for tree graphs. For both simulations a target frequency of $\bar{\omega} = 0.1$ was selected, with intrinsic frequencies $\omega_i = \{-0.6,-0.4,-0.2,0,0.2,0.4,0.6 \}$ chosen to demonstrate a nontrivial frequency convergence.

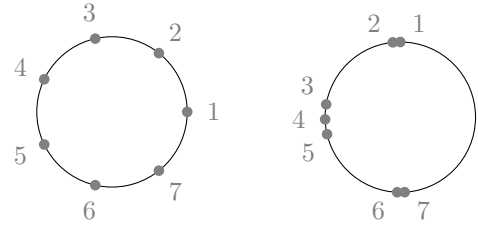
\begin{figure}
  \begin{center}
      \vspace{0cm}
    \subfigure{
    \begin{tikzpicture}[thick,decoration={
    markings,
    mark=at position 0.5 with {\arrow{>}}},
    ] 
        \coordinate (O) at (2,4);
        \def\radius{1cm}
      
        \draw[thin] (O) circle[radius=\radius];
        
        \path (O) ++(0:\radius) coordinate (a);
        \path (O) ++(360/7:\radius) coordinate (b);
        \path (O) ++(720/7:\radius) coordinate (c);
        \path (O) ++(360*3/7:\radius) coordinate (d);
        \path (O) ++(360*4/7:\radius) coordinate (e);
        \path (O) ++(360*5/7:\radius) coordinate (f);
        \path (O) ++(360*6/7:\radius) coordinate (g);
        
        \fill[gray] (a) circle[radius=2pt] ++(0:1em) node {$1$};
        \fill[gray] (b) circle[radius=2pt] ++(360/7:1em) node {$2$};
        \fill[gray] (c) circle[radius=2pt] ++(360*2/7:1em) node {$3$};
        \fill[gray] (d) circle[radius=2pt] ++(360*3/7:1em) node {$4$};
        \fill[gray] (e) circle[radius=2pt] ++(360*4/7:1em) node {$5$};
        \fill[gray] (f) circle[radius=2pt] ++(360*5/7:1em) node {$6$};
        \fill[gray] (g) circle[radius=2pt] ++(360*6/7:1em) node {$7$};

    \end{tikzpicture}}
    \hspace{10pt}
    \subfigure{
        \begin{tikzpicture}[thick,decoration={
    markings,
    mark=at position 0.5 with {\arrow{>}}},
    ] 
        \coordinate (O) at (2,4);
        \def\radius{1cm}
      
        \draw[thin] (O) circle[radius=\radius];
        
        \path (O) ++(90:\radius) coordinate (a);
        \path (O) ++(95.73:\radius) coordinate (b);
        \path (O) ++(170.21:\radius) coordinate (c);
        \path (O) ++(181.67:\radius) coordinate (d);
        \path (O) ++(193.13:\radius) coordinate (e);
        \path (O) ++(267.62:\radius) coordinate (f);
        \path (O) ++(273.35:\radius) coordinate (g);
        
        \fill[gray] (a) circle[radius=2pt] ++(45:1em) node {$1$};
        \fill[gray] (b) circle[radius=2pt] ++(135:1em) node {$2$};
        \fill[gray] (c) circle[radius=2pt] ++(135:1em) node {$3$};
        \fill[gray] (d) circle[radius=2pt] ++(180:1em) node {$4$};
        \fill[gray] (e) circle[radius=2pt] ++(225:1em) node {$5$};
        \fill[gray] (f) circle[radius=2pt] ++(225:1em) node {$6$};
        \fill[gray] (g) circle[radius=2pt] ++(315:1em) node {$7$};
        
    \end{tikzpicture}}
    \end{center}
    \caption{Target formations on the circle for balanced (left) and clustered (right) formations.}
    \label{fig:formations}
    \vspace{0cm}
\end{figure}

\begin{figure}
    \begin{center}
        \hspace{0pt}
        \includegraphics[width=\linewidth]{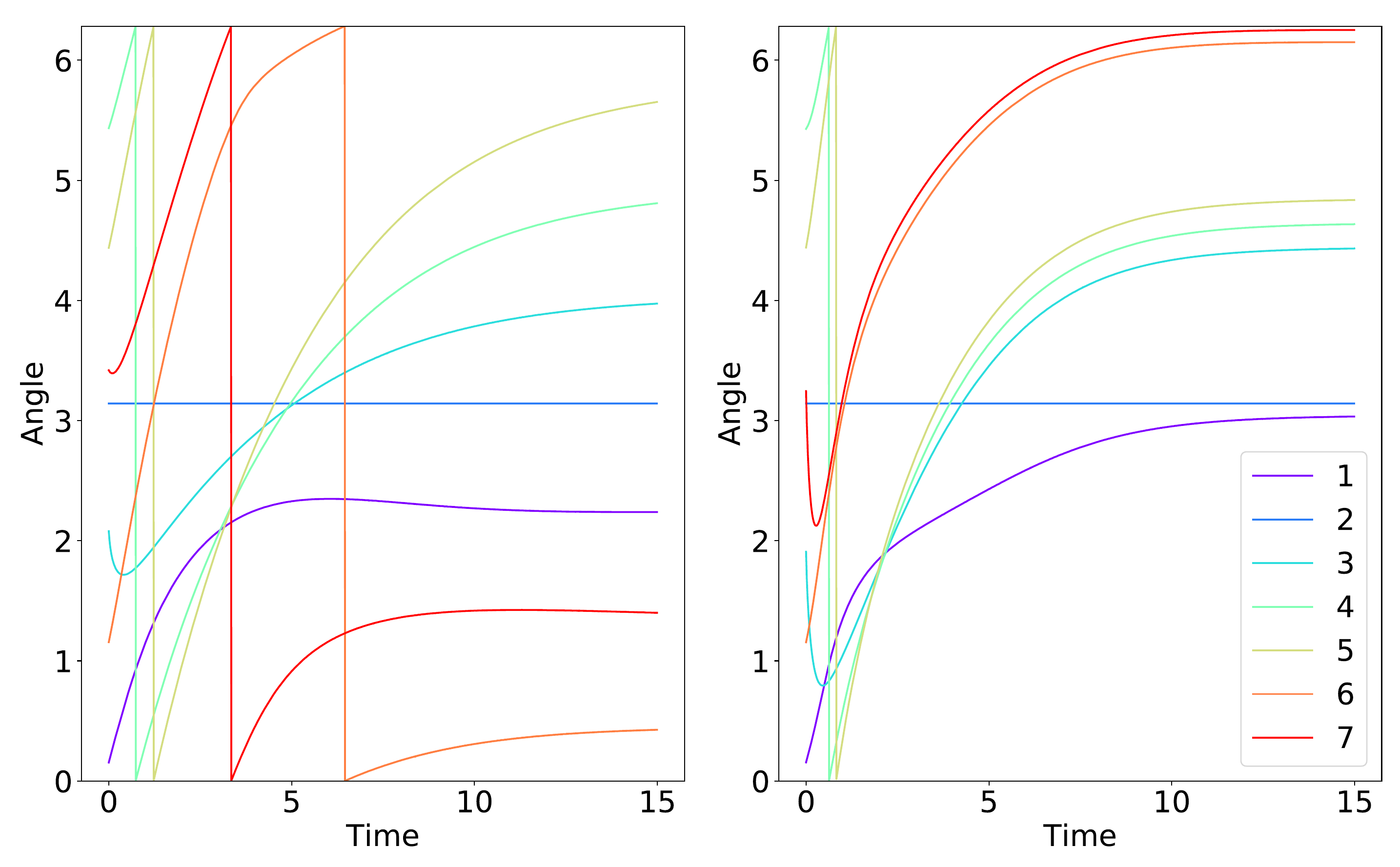}
    \end{center}
    \caption{Position (relative to agent 2) of seven evolving agents for balanced (left) and clustered (right) target formations with dynamics given by the least communication solution.}
    \label{fig:sims}
    \vspace{0cm}
\end{figure}

As shown in Figure \ref{fig:sims}, simulations converge to the target formations with similar convergence times. The coupling coefficients computed for each solution are given below. For the simulations shown, arbitrarily small $\epsilon = 0.01$ was inserted where necessary to ensure bidirectionality, and was confirmed to have a negligible impact on dynamics.

\begin{equation*}
    \left[ \alpha_{ij} \right]^{balanced} = \begin{bmatrix}
        0       & 0     & 0.558 & 0     & 0     & 0     & 0  \\
        0       & 0     & 1.038 & 0     & 0     & 0     & 0  \\
    \epsilon & \epsilon & 0     & 0.623 & 0  & \epsilon & 0  \\
        0       & 0  & \epsilon & 0     & 0.208 & 0     & 0  \\
        0       & 0     & 0     & 0.208 & 0     & 0     & 0  \\
        0       & 0     & 0.068 & 0     & 0     & 0     & \epsilon     \\
        0       & 0     & 0     & 0     & 0     & 1.038 & 0 \\
    \end{bmatrix}
\end{equation*}
\begin{equation*}
    \left[ \alpha_{ij} \right]^{clusters} = \begin{bmatrix}
        0       & 0     & 0.831 & 0     & 0     & 0     & 0 \\
        0       & 0     & 0.658 & 0     & 0     & 0     & 0 \\
    \epsilon & \epsilon & 0     & 2.990 & 0  & \epsilon & 0 \\
        0       & 0  & \epsilon & 0     & 0.997 & 0     & 0 \\
        0       & 0     & 0     & 0.997 & 0     & 0     & 0 \\
        0       & 0     & 0.264 & 0     & 0     & 0     & \epsilon \\
        0       & 0     & 0     & 0     & 0     & 9.991 & 0 \\
    \end{bmatrix}
\end{equation*}

\section{Conclusion}

We have presented a natural generalization of the positivity-based convergence analysis of linear consensus algorithms to target formations on the circle. The approach is based on monotonicity and uses barrier functions to ensure forward invariance of differentially positive dynamics, thereby guaranteeing almost global convergence to limit cycles. By using coupling functions that guarantee convergence of the dynamics, the problem of shaping the collective nonlinear dynamics to a target formation reduces to one of algebra and graph theory. In future work, we hope to relax the barrier conditions on the coupling functions and determine criteria for achieving forward invariance for prescribed sets of initial conditions. Partial relaxations of the monotonicity requirement and analysis of the effects of isolated points of non-smoothness in the dynamics in the absence of forward invariance remain as open problems.

\section{ACKNOWLEDGMENTS}

C.M. was supported by Fitzwilliam College and a Henslow Research Fellowship from the Cambridge Philosophical Society.
G.V.G. was supported by the UCL Centre for Doctoral Training in Data Intensive Science funded by STFC, and by an Overseas Research Scholarship from UCL.
J.G.L. and R.S. acknowledge support from the European Research Council under the Advanced ERC Grant Agreement Switchlet n.670645.


\bibliographystyle{IEEEtran}
\bibliography{references}

\end{document}